\documentclass{amsart}

\usepackage{amsmath}
\usepackage{amsfonts}
\usepackage{amssymb,enumerate}
\usepackage{amsthm}
\usepackage[all]{xy}
\usepackage{hyperref}

\newcommand{\id}{\operatorname{id}}
\newcommand{\fd}{\operatorname{fd}}
\newcommand{\rfd}{\operatorname{Rfd}}

\newcommand{\gfd}{\operatorname{Gfd}}

\newcommand{\CIfd}{\operatorname{CIfd}}

\newcommand{\pd}{\operatorname{pd}}

\newcommand{\Hom}{\operatorname{Hom}}
\newcommand{\Tor}{\operatorname{Tor}}
\newcommand{\Ext}{\operatorname{Ext}}

\newcommand{\Ass}{\operatorname{Ass}}
\newcommand{\amp}{\operatorname{amp}}
\newcommand{\Spec}{\operatorname{Spec}}
\newcommand{\Supp}{\operatorname{Supp}}
\newcommand{\depth}{\operatorname{depth}}

\newcommand{\E}{\operatorname{E}}
\renewcommand{\H}{\operatorname{H}}
\newcommand{\uhom}{{\mathbf R}\Hom}
\newcommand{\utp}{\otimes^{\mathbf L}}

\renewcommand{\H}{\mbox{H}}

\newcommand{\fm}{\frak{m}}
\newcommand{\fp}{\frak{p}}

\newtheorem{thm}{Theorem}[section]
\newtheorem{cor}[thm]{Corollary}
\newtheorem{lem}[thm]{Lemma}
\newtheorem{prop}[thm]{Proposition}

\newtheorem{defn}[thm]{Definition}

\newtheorem{exam}[thm]{Example}
\newtheorem{rem}[thm]{Remark}

\begin{document}

\bibliographystyle{amsplain}

\date{}

\author{Parviz Sahandi, Tirdad Sharif, and Siamak Yassemi}

\address{Department of Mathematics, University of Tabriz, Tabriz,
Iran, and School of Mathematics, Institute for Research in
Fundamental Sciences (IPM), P.O. Box: 19395-5746, Tehran Iran.}

\email{sahandi@ipm.ir}

\address{School of Mathematics, Institute for Research in
Fundamental Sciences (IPM), P.O. Box: 19395-5746, Tehran Iran.}

\email{sharif@ipm.ir}

\address{Department of Mathematics, University of
Tehran, Tehran, Iran and School of Mathematics, Institute for
Research in Fundamental Sciences (IPM), P.O. Box: 19395-5746, Tehran
Iran.}

\email{yassemi@ipm.ir}



\thanks{P. Sahandi was supported in part by a grant from
IPM (No. 89130051).}

\thanks{T. Sharif was supported in part by a grant from IPM (No.
83130311)}
\thanks{S. Yassemi was supported in part by a grant from
IPM (No. 89130213).}

\title{Depth formula via complete intersection flat dimension}

\begin{abstract}

We prove the depth formula, for homologically bounded complexes $X,
Y$ provided that the complete intersection flat dimension of $X$ is
finite and $\sup(X\utp_RY)<\infty$. In particular, let $M$ and $N$
are two $R$-modules and the complete intersection flat dimension of
$M$ is finite. Then $M$ and $N$ satisfies the depth formula,
provided $\Tor^R_i(M,N)=0$ for all $i\ge 1$.


\end{abstract}

\maketitle

\section{Introduction}

In this paper $(R,\fm,k)$ is a local Noetherian ring with unique
maximal ideal $\fm$ and residue field $k$. Two finitely generated
$R$-modules $M$ and $N$ satisfy the \emph{depth formula} if
$$\depth(M\otimes_RN)+\depth R=\depth_RM+\depth_RN.$$

The \emph{Auslander-Buchsbaum formula} asserts that if a finitely
generated $R$-module $M$ has finite projective dimension, then
$\depth_RM+\pd_RM=\depth R$. In \cite{A2} Auslander further
generalized this formula for $M$ as before and $N$ a finitely
generated $R$-module.  In fact he showed that for
$s:=\fd_R(M,N)=\sup\{n|\Tor_n^R(M,N)\neq 0\}$ if either $s=0$ or
$\depth_R\Tor_{s}^R(M,N)\le1$, then
$$(*) \qquad s=\depth R-\depth_RM-\depth_RN+\depth_R\Tor_s^R(M,N).$$
The case $s=0$ is the depth formula. Note that with $N=k$, the
residue field of $R$, the equality $(*)$ is just the
Auslander-Buchsbaum formula. Also note that for any finitely
generated $R$-modules $M$ and $N$, if $\fd_R(M,N)<\infty$, then all
the terms in $(*)$ are defined and finite.

In \cite{HW} Huneke and Wiegand showed that for the complete
intersection ring $R$ and non-zero finitely generated $R$-modules
$M$ and $N$ if $\Tor^R_i(M,N) = 0$ for all $i\geq1$, then $M$ and
$N$ satisfies the depth formula.

Later, Araya and Yoshino \cite{AY} showed that for finitely
generated $R$-modules $M$ and $N$ such that $M$ has finite complete
intersection dimension \cite{AGP} and $s=\fd_R(M, N)<\infty$, if
$s=0$ or $\depth_R \Tor^R_s(M,N)\le 1$, then $(*)$ holds for $M$ and
$N$.

Then, Choi and Iyengar \cite{CI} showed that for finitely generated
$R$-modules $M$ and $N$ such that $s=\fd_R(M,N)<\infty$, if $M$ has
finite complete intersection dimension \cite{AGP}, then $\depth_R
M+\depth_R N-\depth R\ge s$ with equality if and only if
$\depth_R\Tor^R_s(M, N)=0$.

Finally Bergh and Jorgensen \cite{BJ}, prove that the depth formula
holds for modules $M$ and $N$ such that $\fd_R(M,N)=0$ in certain
cases over a Cohen-Macaulay local ring, provided one of the modules
has reducible complexity.

In Section 2 of this paper we recall the definition of complete
intersection flat dimension for complexes. We show that both
Gorenstein flat dimension and large restricted flat dimension are
refinements of the complete intersection flat dimension.

Section 3 which is the main part of this paper is devoted to the
depth formula for complexes. More precisely let $(R,\fm,k)$ be a
local ring and let $X, Y\in{\mathcal D}_{b}(R)$ such that the
complete intersection flat dimension of $X$ is finite and
$\sup(X\utp_RY)<\infty$. Then
$$\depth_R(X\utp_RY)=\depth_RX+\depth_RY-\depth R,$$
which is Theorem \ref{DF}. One of the main tool in our study is
Proposition \ref{split}, which is proved by using the universal
resolutions. From a general point of view, the construction of the
universal resolutions is based on the bar construction, a nice
structure that comes from algebraic topology. In \cite{I97} Iyengar
studied the bar construction and consequently, the universal
resolutions in commutative algebra for arbitrary modules. We refer
the interested reader to \cite[Section 3]{Av} for details. In
Theorem \ref{dep} we show that for $X, Y\in{\mathcal D}_{b}(R)$ such
that the complete intersection flat dimension of $X$ is finite and
$\sup(X\utp_RY)<\infty$, we have
$$
\sup(X\utp_{R}Y)=\sup\{\depth R_\fp-\depth_{R_\fp}
{X_\fp}-\depth_{R_\fp} {Y_\fp}|\fp\in\Supp X\cap\Supp Y\}.
$$
We end the paper with some results about $\inf\uhom_R(X,Y)$.

To facilitate the reading of the introduction and of the paper, we
first review some basic facts on complexes from \cite{F}, \cite{F79}
and \cite{AF}.

Let $X$ be a complex of $R$-modules and $R$-homomorphisms. For an
integer $n$, the $n$-th \emph{shift} or \emph{suspension} of $X$ is
the complex $\Sigma^{n}X$ with $(\Sigma^{n}X)_{\ell}=X_{\ell-n}$ and
$\partial_{\ell}^{\Sigma^{n}X}=(-1)^{n}\partial_{\ell-n}^{X}$ for
each $\ell$. 
The \emph{supremum} and the \emph{infimum} of a complex $X$, denoted
by $\sup(X)$ and $\inf(X)$ are defined by the supremum and infimum
of $\{i\in\mathbb{Z}|\H_i(X)\neq0\}$ and set
$\amp(X)=\sup(X)-\inf(X)$.

The symbol ${\mathcal D}(R)$ denotes the \emph{derived category} of
$R$-complexes. The full subcategories ${\mathcal D}_{-} (R)$
${\mathcal D}_{+} (R)$, ${\mathcal D}_{b} (R)$ and ${\mathcal D}_{0}
(R)$ of ${\mathcal D}(R)$ consist of $R$-complexes $X$ while
$\H_{\ell}(X)=0$, for respectively $\ell\gg 0$, $\ell\ll 0$,
$|\ell|\gg 0$ and $\ell\neq 0$. By ${\mathcal D}^f(R)$ we denote the
full subcategory consisting of complexes $X$ with all homology
modules ${\H}_{\ell}(X)$ are finitely generated over $R$, called
\emph{homologically degreewise finite} complexes. A complex is
\emph{homologically finite} if it is homologically both bounded and
degreewise finite. The right derived functor of the homomorphism
functor of $R$-complexes and the left derived functor of the tensor
product of $R$-complexes are denoted by $\uhom_{R}(-,-)$ and
$-\utp_{R}-$, respectively. A homology isomorphism is a morphism
$\alpha:X\to Y$ such that $\H(\alpha)$ is an isomorphism; homology
isomorphisms are marked by the sign $\simeq$, while $\cong$ is used
for isomorphisms. The equivalence relation generated by the homology
isomorphisms is also denoted by $\simeq$. Let $X$ be an $R$-complex.
The \emph{support} of $X$, denoted by $\Supp(X)$, is defined as
$$
\Supp(X)=\{\fp\in\Spec(R)|\H(X_{\fp})\neq0\}.
$$
If $X\in{\mathcal D}_{-}(R)$, the \emph{depth} of $X$ is defined by
$$\depth_RX=-\sup\uhom_R(k,X).$$
See \cite{FI}, for several equivalent conditions of the depth of a
complex.

\section{Basic notions and results}

In this section we recall the definition of the complete
intersection flat dimension, and give some result that we will used
in the next section. The ideal $J$ of $R$ is called a \emph{complete
intersection ideal}, if $J$ is generated by an $R$-regular elements.
We say that $R$ has a \emph{deformation} if there exists a local
ring $Q$ and a complete intersection ideal $J$ in $Q$ such that
$R=Q/J$. A \emph{quasi-deformation} of $R$ is a diagram of local
homomorphisms $R\rightarrow R'\leftarrow Q$, with $R\rightarrow R'$
a flat extension and $R'\leftarrow Q$ a deformation. If the kernel
of $Q\to R'$ is generated by a $Q$-regular sequence of length $c$,
we will sometimes say that the quasi-deformation $R\rightarrow
R'\leftarrow Q$ has \emph{codimension $c$}.

The following definition is from \cite{SSY} and \cite{SW}.

\begin{defn} For each homologically bounded $R$-complex $X$, the
\emph{complete intersection flat dimension} of $X$ is defined by
$$
\CIfd_RX:=\inf\{\fd_Q(R'\utp_RX)-\fd_QR'| \text{ }R\rightarrow
R'\leftarrow Q \text{ is a }\text{ quasi-deformation}\}.$$
\end{defn}

Now we recall the large restricted flat dimension of complexes
and prove that it refines the complete intersection flat
dimension. The \emph{large restricted flat dimension} of $X$ over
$R$, as introduced in \cite{CFF}, is the quantity
$$
\rfd_RX:=\sup\{\sup(F\utp_{R}X)|F\text{ and }R\text{-module with
}\fd_RF<\infty\}.
$$
This number is  finite, as long as $\H(X)$ is nonzero; see
\cite[Proposition 2.2]{CFF}. It is also given \cite[Theorem
2.4(b)]{CFF} by the following alternate formula also known as
Chouinard's \cite{C} formula:
$$
\rfd_RX:=\sup \{\depth  R_\fp - \depth _{R_{\fp}} X_\fp | \fp\in
\Spec (R)\}.
$$
We show that the large restricted flat dimension is a refinement of
the complete intersection flat dimension.

\begin{prop} \label{CH} Let $X$ be a homologically bounded $R$-complex.
Then we have $\rfd_RX\le\CIfd_RX$, with equality if, $\CIfd_RX$ is
finite. In this case we have
$$\CIfd_RX=\sup \{\depth  R_\fp - \depth _{R_{\fp}} X_\fp | \fp\in
\Spec (R)\}.$$
\end{prop}
\begin{proof} One checks easily that for a quasi-deformation $R\rightarrow S\leftarrow
Q$, we have $\rfd_Q (X\utp_{R} S)-\rfd_Q S=\rfd_RX$. Now the
inequality follows easily from the definition of the complete
intersection flat dimension. The last equality follows from
\cite[Theorem 2.4(b)]{CFF}.
\end{proof}

The proof of the following proposition is easy, so we omit it.

\begin{prop}\label{P} Let $X$ be a homologically bounded $R$-complex.
For each prime ideal $\fp \in \Spec(R)$ there is an inequality
$$\CIfd_{R_{\fp }} {X_{\fp}}\le\CIfd_RX.$$
\end{prop}

The \emph{Gorenstein flat dimension} of each homologically bounded
complex of $R$-modules $X$ is defined by
$$
\gfd_RX:=\inf\{\sup\{n|G_n\neq0\}|G\text{ is a Gorenstein flat
resolution of }X\}.
$$
The reader may consult \cite{C} or the book of Enochs and Jenda
\cite{bookEJ} for details.

Let $X$ be a homologically bounded $R$-complex. Then by using some
known result and \cite[Theorem 2.4]{SSY} there is the following
sequence of inequalities
$$\rfd_RX\le\gfd_RX\le\CIfd_RX\le\fd_RX,$$
with equality to the left of any finite number.

\section{Depth formula}

In this section we prove a depth formula for complexes of finite
complete intersection flat dimension.

Let $Q$ be a deformation of codimension one of $R$. Suppose that $M$
is an $R$-module. Let $A$ be the \emph {Koszul} complex resolving
$R$ over $Q$ with augmentation $\kappa:A\longrightarrow R$ and the
structure map $\eta^A:Q\longrightarrow A$. It follows from
\cite[(2.2.7)]{Av} that $M$ has a $Q$-projective resolution $U$,
such that it is a DG module over $A$. Now from \cite[(3.1.1)]{Av} it
is easy to see that $M$ has an $R$-projective resolution $F(U)$,
such that the following exact sequence of $R$-complexes is
degreewise split
\begin{displaymath}
\xymatrix{ 0 \ar[r]  & R\otimes_Q U \ar[r]& F(U) \ar[r]  & \Sigma^2
F(U) \ar[r] & 0,}
\end{displaymath}
in which
$$ F_n(U)=\displaystyle{\bigoplus_{p=0}^{n}}  \bar{A}_1^{(p)}\otimes_R
\bar{U}_{n-2p},$$
where $\bar{A}_1^{(p)}=\bar{A_1}\otimes_R\ldots
\otimes_R\bar{A_1}$ ($p$ times) and
\begin{align*}
 & \partial_n^{F(U)}(\bar{a}_1\otimes \bar{a}_2\otimes \ldots \otimes
\bar{a}_p\otimes \bar{u}) \\[1ex]
= & \bar{a}_1\otimes \bar{a}_2\otimes \ldots \otimes
\bar{a}_p\otimes
\partial^{\bar{U}}(\bar{u})-\bar{a}_1\otimes \bar{a}_2\otimes \ldots \otimes
\bar{a}_{p-1}\otimes \bar{a}_p\bar{u}.
\end{align*}
In \cite[(5.10)]{Yo} Yoshino generalized the above exact sequence,
for bounded complexes with finite homology modules. His method is
based on constructing the Eisenbud resolutions for complexes. This
exact sequence plays an important role in this paper. In the next
result we extend it to bounded complexes, which is the key result
for the proof of the depth formula. For our main purpose, we only
need to show that if $X\in{\mathcal D}_{b}(R)$, then there is a
$Q$-projective resolution $W$ of $X$ and an $R$-projective
resolution $F(W)$ of $X$, such that the following sequence of
$R$-complexes is degreewise split
\begin{displaymath}
\xymatrix{ 0 \ar[r]  & R\otimes_Q W \ar[r]& F(W) \ar[r]  & \Sigma^2
F(W) \ar[r] & 0.}
\end{displaymath}
This is the content of the following proposition.

\begin{prop}\label{split} Let $Q$ be a deformation of
codimension one of $R$ and let $X\in{\mathcal D}_{b}(R)$. Then there
exists a $Q$-projective resolution $W$ of $X$ and an $R$-projective
resolution $F(W)$ of $X$, such that, the following sequence of
$R$-complexes is degreewise split
\begin{displaymath}
\xymatrix{ 0 \ar[r]  & R\otimes_Q W \ar[r]& F(W) \ar[r]  & \Sigma^2
F(W) \ar[r] & 0.}
\end{displaymath}
\end{prop}

\begin{proof} Let $P\simeq X$ be an $R$-projective
resolution of $X$, and set $t=\inf X=\inf P$. We will induct on
$\amp(P)$. If $\amp(P)=0$ then $P\simeq\Sigma^{t}\H_{t}(P)$. So that
$X$ is a shift of an $R$-module. Now the assertion follows from the
exact sequence before the proposition.

Next suppose that $\amp(P)>0$. Consider the following truncations of
$P$
\begin{displaymath}
\xymatrix{ P_{\leq t}=0 \ar[r]  & 0 \ar[r]& P_t \ar[r] &
P_{t-1} \ar[r] & \cdots,\\
P_{>t}=\cdots \ar[r]  & P_{t+2} \ar[r]& P_{t+1} \ar[r] & 0\ar[r] &
0.}
\end{displaymath}
Thus we have the following exact sequence, which is degreewise split
of $R$-complexes, and so it is degreewise split as $Q$-complexes.
\begin{displaymath}
\xymatrix{ 0 \ar[r]  & P_{\leq t} \ar[r]& P \ar[r]  & P_{>t} \ar[r]
& 0.}
\end{displaymath}
It is clear that $\amp(P_{\leq t})$ and $\amp(P_{>t})$ are less than
$\amp(P)$. Let $V$ and $U$ be $Q$-projective resolutions of $\Sigma
P_{\leq t}$ and $P_{>t}$, respectively. Note that the differential
$\partial_{t+1}^P: P_{t+1}\rightarrow P_t$ induces an $R$-linear
(and so $Q$-linear) morphism $\gamma: P_{>t}\rightarrow \Sigma
P_{\leq t}$. Hence there is a $Q$-linear morphism $\beta:
U\rightarrow V$ of $Q$-projective resolutions. By \cite[(1.24)]{F}
or \cite[Section 1.5]{Weib} there is an exact sequence of
$Q$-complexes as the following
\begin{displaymath}
\xymatrix{ 0 \ar[r]  & V \ar[r]& C(\beta) \ar[r]  & \Sigma U \ar[r]
& 0,}
\end{displaymath}
in which $C(\beta)$ is the mapping cone of $\beta$. Since the
mapping cone of $\gamma$ is $\Sigma P$, it follows that $C(\beta)$
is a $Q$-projective resolution of $\Sigma P$. By the induction
hypothesis there are $R$-projective resolutions $F(V)$ and $F(U)$
for $\Sigma P_{\leq t}$ and $P_{>t}$, respectively, such that, the
following sequences of $R$-complexes are degreewise split
\begin{displaymath}
\xymatrix{ 0 \ar[r]  & R\otimes_Q V \ar[r]& F(V) \ar[r]  &
\Sigma^{2} F(V) \ar[r] & 0,\\ 0 \ar[r]  & R\otimes_Q U \ar[r]& F(U)
\ar[r]  & \Sigma^{2} F(U) \ar[r] & 0.}
\end{displaymath}
Now define the morphism $F(\beta): F(U)\rightarrow F(V)$ by
$$
F(\beta)_n(\bar{a}_1\otimes \bar{a}_2\otimes
\ldots \otimes \bar{a}_p\otimes \bar{u}_{n-2p})=\bar{a}_1\otimes
\bar{a}_2\otimes \ldots \otimes \bar{a}_p\otimes
\bar{\beta}_{n-2p}(\bar{u}_{n-2p}).
$$
It is easy to see that $F(\beta)$ is a morphism of $R$-complexes.
Using the exact sequence
\begin{displaymath}
\xymatrix{ 0 \ar[r]  & F(V) \ar[r]& C(F(\beta)) \ar[r]  & \Sigma
F(U) \ar[r] & 0,}
\end{displaymath}
we get that $C(F(\beta))$ is an $R$-projective resolution of $\Sigma
P$. Now by considering the induction hypothesis we can see that the
following diagram is commutative with exact rows
\begin{displaymath}
\xymatrix{ 0 \ar[r]  & R\otimes_Q U \ar[r] \ar[d]^{R\otimes_Q\beta}
& F(U) \ar[r] \ar[d]^{F(\beta)}& \Sigma^{2} F(U) \ar[r]
\ar[d]^{\Sigma^2F(\beta)}& 0\\
0 \ar[r]  & R\otimes_Q V \ar[r]& F(V) \ar[r]  & \Sigma^{2} F(V)
\ar[r] & 0.}
\end{displaymath}
The above exact sequences induces the following exact sequence
\begin{displaymath}
\xymatrix{ 0 \ar[r]  & C(R\otimes_Q\beta) \ar[r]& C(F(\beta)) \ar[r]
& \Sigma^2 C(F(\beta)) \ar[r] & 0.}
\end{displaymath}
It is also straightforward to see that $F(C(\beta))=C(F(\beta))$.
Therefore, the following sequence is exact
\begin{displaymath}
\xymatrix{ 0 \ar[r]  & R\otimes_Q C(\beta) \ar[r]& F(C(\beta))
\ar[r] & \Sigma^2 F(C(\beta)) \ar[r] & 0.}
\end{displaymath}
Now setting $W=\Sigma^{-1}C(\beta)$ we get the desired exact
sequence.
\end{proof}

\begin{lem}\label{sur} Let $Q$ be a deformation of
codimension $c$ of $R$ and let $X, Y\in{\mathcal D}_{b}(R)$, such
that $\lambda=\sup(X\utp_RY)<\infty.$ Then
\begin{itemize}
\item[(a)] $\sup(X\utp_QY)=\lambda+c.$
\item[(b)] $\depth_R(X\utp_RY)=\depth_Q(X \utp_Q Y)+c.$
\end{itemize}
\end{lem}

\begin{proof} (a) We will induct
on $c$. Suppose that $c=1.$ Using Proposition \ref{split} it follows
that there is the following degreewise split sequence of complexes
of projective $R$-modules
\begin{displaymath}
\xymatrix{ \Gamma:= 0 \ar[r]  & R\otimes_Q W \ar[r]& F(W) \ar[r] &
\Sigma^{2}F(W) \ar[r] & 0,}
\end{displaymath}
where $W$ and $F(W)$, respectively are $Q$ and $R$-projective
resolutions of $X$. As the above exact sequence is degreewise split,
we get that $\Gamma\otimes_R Y$ is an exact sequence of
$R$-complexes. Thus taking homology of $\Gamma\otimes_R Y$, we find
that $\sup(X\utp_QY)=\lambda+1.$ Now suppose that $c>1$. Let
$R_1=Q/(x_1, x_2,\cdots,x_{c-1})$ and $R=R_1/(x_c)R_1$, where $x_1,
x_2,\cdots,x_c$ is a $Q$-regular sequence. Thus by the induction
hypothesis we have $\sup(X \utp_{R_1}Y)=\lambda+1$ and $\sup(X
\utp_Q Y)=\sup(X \utp_{R_1}Y)+c-1$. Hence the proof of (a) is
complete combining the last equalities.

(b): since
$Q\to R$ is surjective, it is known that $\depth_Q(X\utp_R
Y)=\depth_R(X\utp_R Y)$, see \cite[Proposition 5.2(1)]{I}. We claim
that $\depth_R(X\utp_RY)$ and $\depth_Q(X\utp_QY)$ are
simultaneously finite numbers. Indeed, consider the following exact
sequence of complexes
$$0\to X\utp_Q Y\to X\utp_R Y\to \Sigma^2(X\utp_R Y)\to
0.$$ Let $\depth_R(X\utp_RY)$ be finite. Then $\depth_Q(X\utp_R Y)$
is finite. By part (a) we have $\sup(X\utp_QY)<\infty$ and so
$\depth_Q(X\utp_Q Y)>-\infty$, see for example \cite[Section 2,
Observation (3)]{I}. By \cite[Proposition 5.1]{I} we have that $\depth_Q(X\utp_Q Y)$ is finite.
Now assume that $\depth_Q(X\utp_Q Y)$ is finite. By our assumption
$\sup(X\utp_RY)<\infty$. Therefore $\depth_Q(X\utp_R Y)>-\infty$. Again by \cite[Proposition 5.1]{I} we have that $\depth_Q(X\utp_R Y)$ is finite.

Suppose $c=1$. Since $\Gamma$ is degreewise split,
$\Gamma\otimes_{R} Y$ is an exact sequence of complexes. Now from
\cite[(6.49)]{F} we get
$$\sup\uhom_{Q}(\ell, X \utp_{R} Y)=\sup\uhom_{Q}(\ell,X \utp_Q Y)-1,$$  where $\ell$ is the residue field of $Q$.
It follows that $\depth_{Q}(X\utp_{R} Y)=\depth_Q(X \utp_Q Y)+1$.
Therefore
$$\depth_{R}(X\utp_{R} Y)=\depth_Q(X \utp_Q Y)+1.$$
If $c>1$, using the notations of part (a) we find that
$$
\depth_{R}(X \utp_{R} Y)=\depth_{R_1}(X \utp_{R_1} Y)+1,
$$
$$
\depth_{R_1}(X \utp_{R_1}Y)=\depth_Q(X \utp_Q Y)+c-1.
$$
Now by joining the last equalities, the proof is complete.
\end{proof}

Now we are in the position of proving \emph{depth formula} for the
complete intersection flat dimension of complexes.

\begin{thm}\label{DF} Let $R$ be a ring and let $X, Y\in{\mathcal
D}_{b}(R)$ such that $\CIfd_RX<\infty$ and
$\sup(X\utp_RY)<\infty$. Then
$$\depth_R(X\utp_RY)=\depth_RX+\depth_RY-\depth R.$$
\end{thm}

\begin{proof} Since $\CIfd_RX<\infty$, so
that there is a codimension $c$ quasi-deformation, $R\rightarrow
R'\leftarrow Q$, such that $\CIfd_RX=\fd_Q(R'\utp_RX)-c$. We have
the isomorphisms $X'\utp_{R'}Y'=(X\utp_RY)'$, where $W'$ stands
for $R'\utp_RW$, for an $R$-complex $W$. Since $R'$ is a
faithfully flat $R$-module, then $\sup(X'\utp_{R'}Y')<\infty$.
Thus by Lemma \ref{sur}(b) we have the following equality
$$\depth_{R'}(X'\utp_{R'} Y')=\depth_Q(X' \utp_Q Y')+c.$$
Since $\fd_{Q}X'<\infty$, by \cite[(12.8)]{F} (or \cite[Theorem
4.1]{I}) we have
$$\depth_Q(X' \utp_Q Y')=\depth_QY'+\depth_QX'-\depth Q.$$
On the other hand, by \cite[Corollary 2.6]{I} we find the following
equalities
\begin{align*}
\depth R'= & \depth R+\depth R'/\fm R', \\[1ex]
\depth_{R'}X'= & \depth_RX+\depth R'/\fm R', \\[1ex]
\depth_{R'}Y'= &\depth_RY+\depth R'/\fm R', \\[1ex]
\depth_{R'}(X'\utp_{R'} Y')= &\depth_{R}(X\utp_{R} Y)+\depth R'/\fm
R'.
\end{align*}
By combining the above equalities we get
$$
\depth_R(X\utp_RY)=\depth_RX+\depth_RY-\depth R,
$$
as desired.
\end{proof}

The following corollary generalizes a result of Auslander \cite{A2}
and improves a result of Huneke and Weigand \cite[Proposition
2.5]{HW} where the ring is assumed to be complete intersection and
the modules $M$ and $N$ finitely generated. Also it generalizes
results of Araya and Yoshino \cite[Theorem 2.5]{AY} and Iyengar
\cite[Theorem 4.3]{I}.

\begin{cor} \label{Huneke}
Let $R$ be a ring and let $M,N$ be two $R$-modules such that
$\Tor_i^R(M,N)=0$ for all $i\geq1$. If $\CIfd_RM<\infty$, then
$$
\depth_R(M\otimes_RN)=\depth_RM+\depth_RN-\depth R.
$$
\end{cor}

\begin{proof} By the hypothesis we have $\sup(M\utp_RN)=\sup\{i|\Tor_i ^R (M,N)\neq
0\}=0$. Thus we have $M\utp_RN\simeq M\otimes_RN$. Now the result
is clear from Theorem \ref{DF}.
\end{proof}

The next theorem is the \emph{dependency formula} \cite{JJ} for the
complete intersection flat dimension.

\begin{thm}\label{dep} Let $R$ be a ring and let
$X, Y\in{\mathcal D}_{b}(R)$ such that $\sup(X\utp_RY)<\infty$ and
$\CIfd_RX<\infty$. Then
$$
\sup(X\utp_{R}Y)=\sup\{\depth R_{\fp}-\depth_{R_\fp}
{X_\fp}-\depth_{R_\fp} {Y_\fp}|\fp\in\Supp X\cap\Supp Y\}.
$$
\end{thm}

\begin{proof} Let $\fp\in\Supp X\cap\Supp Y$. Thus by Proposition \ref{P} we have
$\CIfd_{R_{\fp }} {X_{\fp}}<\infty$. Since
$\sup(X_{\fp}\utp_{R_{\fp}}Y_{\fp})<\infty$, we have the following
chain of (in)equalities, where the equality holds by Theorem
\ref{DF}, and the first inequality holds by \cite[(2.7)]{FI}.
\begin{align*}
\depth{R_\fp}-\depth{X_\fp}-\depth_{R_\fp} {Y_\fp}
       = & -\depth_{R_\fp}(X_\fp\utp_{R_\fp}Y_\fp) \\[1ex]
       \le & \sup(X_\fp\utp_{R_{\fp}}Y_\fp) \\[1ex]
       \le & \sup(X\utp_{R}Y) \\[1ex]
       < & \infty,
\end{align*}
and we have the equality if $\fp\in\Ass(\H_{s}(X\utp_{R}Y))$, where
$s=\sup(X\utp_{R}Y)$.
\end{proof}

The following corollary generalizes \cite[Theotem 3]{CI}.

\begin{cor} Let $R$ be a ring and let
$X, Y\in{\mathcal D}_{b}(R)$ such that $\CIfd_RX<\infty$. If
$s=\sup(X\utp_RY)<\infty$, then
$$
\depth R-\depth_RX-\depth_RY\leq s.
$$
with equality if and only if $\depth_R\H_s(X\utp_RY)=0$.
\end{cor}

\begin{cor} Let $X, Y\in{\mathcal D}_{b}(R)$ such that $\CIfd_RX<\infty$. Then the
following are equivalent:
\begin{itemize}
\item[(a)] $\sup(X\utp_RY)<\infty$.

\item[(b)] $\sup(X\utp_RY)\le\CIfd_RX$.
\end{itemize}
\end{cor}

\begin{proof} If for all integer $n$, $\H_n(X\utp_RY)=0$, then the assertion
holds. So assume that $s=\sup(X\utp_RY)<\infty$. Thus by Theorem
\ref{dep}, $s=\depth
R_{\fp}-\depth_{R_{\fp}}X_{\fp}-\depth_{R_{\fp}}Y_{\fp}$ for some
$\fp\in\Supp(X)\cap\Supp(Y)$. Choose an integer
$n>\CIfd_RX=\rfd_RX\geq\depth R_{\fp}-\depth_{R_{\fp}}X_{\fp}\geq
s$. Then $\H_n(X\utp_RY)=0$.
\end{proof}

Now we turn our attention to measure $\inf\uhom_R(X,Y)$.

\begin{thm}\label{3.13} Let $(R,\fm,k)$ be a complete local ring and let
$X\in\mathcal{D}_b(R)$. If $Y\in\mathcal{D}_b^f(R)$ such that
$-\inf\uhom_R(X,Y)<\infty$ and $\CIfd_RX<\infty$, then
$$-\inf\uhom_R(X,Y)=\depth R-\depth_R X-\inf Y.$$
\end{thm}

\begin{proof}  Let $E:=\E(k)$ be the injective hull of $k$. Using completeness of $R$
and \cite[Page 180]{CH}, we have
$$
\uhom_R(\uhom_R(Y,E),E)=Y\utp_R\uhom(E,E)=Y.
$$
Therefore we obtain that
\begin{align*}
-\inf(\uhom_R(X,Y))= & -\inf(\uhom_R(X,\uhom_R(\uhom_R(Y,E),E))) \\[1ex]
       = & -\inf(\uhom_R(X\utp_R\uhom_R(Y,E),E)) \\[1ex]
       = & \sup(X\utp_R\uhom_R(Y,E)).
\end{align*}
Since $\Supp X\cap\Supp(\uhom_R(Y,E))=\{\fm\}$, using Theorem
\ref{dep}(a) we obtain
\begin{align*}
-\inf(\uhom_R(X,Y))= & \sup(X\utp_R\uhom_R(Y,E)) \\[1ex]
       = & \depth R-\depth_RX-\depth_R\uhom(Y,E) \\[1ex]
       = & \depth R-\depth_RX-\inf Y,
\end{align*}
where the last equality holds by \cite{Y1}.
\end{proof}

\begin{cor} Let $(R,\fm)$ be a complete local ring and $M$ be an
$R$-module such that $\CIfd_RM<\infty$. If $N$ is a finite
$R$-module, then the following are equivalent:
\begin{itemize}
\item[(a)] $\Ext_R^n(M,N)=0$ $n\gg0$.

\item[(b)] $\Ext_R^n(M,N)=0$ $n>\depth R-\depth_RM$.
\end{itemize}
\end{cor}

\begin{proof} It is easily follows from Theorem \ref{3.13} and noting that we have the equality
$-\inf\uhom_R(M,N)=\sup\{i|\Ext^i_R(M,N)\neq0\}.$
\end{proof}

\begin{rem}\label{A} {\rm (1) Let $R$ be a ring and let
$X, Y\in{\mathcal D}_{b}(R)$ such that $\sup(X\utp_RY)<\infty$,
$\gfd_RX<\infty$ and $\id_RY<\infty$. Then
$$
\sup(X\utp_{R}Y)=\sup\{\depth R_{\fp}-\depth_{R_\fp}
{X_\fp}-\depth_{R_\fp} {Y_\fp}|\fp\in\Supp X\cap\Supp Y\}.
$$
To see this let $\fp\in\Supp(X)\cap\Supp(Y)$. Then using the
hypothesis, \cite[Lemma 5.1.3]{C} and \cite[Proposition 5.3.I]{AF}
respectively we have $\gfd_{R_{\fp}}X_{\fp}$ and
$\id_{R_{\fp}}Y_{\fp}$ are finite numbers. Thus from \cite [Theorem
6.3(i)]{CH} we have
$$
\depth{R_\fp}-\depth{X_\fp}-\depth_{R_\fp}
{Y_\fp}=-\depth_{R_\fp}(X_\fp\utp_{R_\fp}Y_\fp).
$$
The rest of the argument is the same as proof of Theorem \ref{dep}.

(2) Immediately from (1) we have if $X$ is a bounded $R$-complex
such that $\gfd_RX<\infty$, then
$$\sup(X\utp_{R}E)=\depth R-\depth_RX,$$ where $E:=\E(k)$ is the injective hull of
$k$ (cf. \cite[Theorem (8.7)]{IS})

(3) Let $(R,\fm,k)$ be a complete local ring and let
$X\in\mathcal{D}_b(R)$. If $Y\in\mathcal{D}_b^f(R)$ such that
$-\inf\uhom_R(X,Y)<\infty$, $\gfd_RX<\infty$ and $\fd_RY<\infty$,
then
$$-\inf\uhom_R(X,Y)=\depth R-\depth_R X-\inf Y.$$
Indeed in proof of Theorem \ref{3.13} we have seen that
$-\inf(\uhom_R(X,Y))=\sup(X\utp_R\uhom_R(Y,E))$. Since
$\fd_RY<\infty$, we have $\id_R\uhom_R(Y,E)<\infty$. Now using (1)
the proof is complete.}
\end{rem}

The following example shows that the completeness assumption of $R$
is crucial in Theorem \ref{3.13} and Remark \ref{A}(3).

\begin{exam}
Let $(R,\fm)$ be a local domain which is not complete with respect
to the $\fm$-adic topology. In \cite [(3.3)]{AER} it is shown that
$\Hom_R(\widehat{R},R)=0$. Therefore, when $X=\widehat{R}$ and
$Y=R$.  It is clear that $\depth R-\depth_RX=0$ but
$-\inf\uhom_R(X,Y)=\sup\{i|\Ext^i_R(X,Y)\neq0\}\neq 0$.
\end{exam}

\section*{Acknowledgments}

The authors are grateful to Yuji Yoshino for thoughtful discussions
about this research. They also would like to thank the referee for
his/her useful comments.

\end{document}